\documentclass[11pt]{article}  
\usepackage{amsmath}
\usepackage{amssymb}
\usepackage{theorem}
\usepackage{euscript}
\usepackage{pstricks}

\newtheorem{theorem}{Theorem}[section]
\newtheorem{lemma}[theorem]{Lemma}

\numberwithin{equation}{section}
\title{\sffamily Whitney's theorem for local anisotropic polynomial $L_p$-approximation, $0<p<1$
}
\author{Dinh D\~ung$^a$\footnote{Corresponding author. Email: dinhzung@gmail.com}, Nguyen Van D\~ung$^b$ 
and Nguyen Dinh Hoa$^c$ \\[2ex]
$^{a,c}$ Vietnam National University, Hanoi, Information Technology Institute \\
144, Xuan Thuy, Hanoi, Vietnam  \\[1.5ex]
$^b$University of Transport and Communications\\
Lang Thuong, Dong Da, Hanoi, Vietnam}
\date{\ttfamily May 27, 2013 -- Version 0.2}
\def\II{{\mathbb I}}

\def\ZZ{{\mathbb Z}}
\def\NN{{\mathbb N}}
\def\RR{{\mathbb R}}
\def\Pp{{\mathcal P}}

\newlength{\fixboxwidth}
\begin{document}
\maketitle

\begin{abstract}
Dinh D\~ung and T. Ullrich have proven a multivariate Whitney's theorem for
the local anisotropic  polynomial approximation in $L_p(Q)$ for $1 \le p \le \infty$,
where $Q$ is a $d$-parallelepiped in $\RR^d$ with sides parallel to the coordinate axes. They considered the 
error of best approximation of a function $f$ by algebraic polynomials of fixed degree at most $r_i - 1$ in variable $x_i,\ i=1,...,d$. The convergence rate of the approximation error when the size of $Q$ going to $0$ is characterized by a so-called total mixed modulus of smoothness. The method of proof used by these authors is not suitable to the case 
$0 <p<1$. In the present paper, by a different method we proved this theorem for $0< p \le \infty$.

\medskip
\noindent
{\bf Keywords}\ Whitney's theorem; Anisotropic approximation by polynomials; 
Total mixed modulus of smoothness; Marchaud's inequality. 

\medskip
\noindent
{\bf Mathematics Subject Classifications (2010)} \ 41A10; 41A50; 41A63.   
\end{abstract}

\section{Introduction}

Let $\omega_r(f,\cdot)_{p,I}$ be the $r$th modulus of smoothness of  a function $f \in L_p(I)$, and $E_r(f)_{p,I}$ is the error of best $L_p$-approximation $E_r(f)_{p,I}$ of 
$f$ by algebraic polynomials of degree at most $r-1$, where $I:=[a,b]$ is an interval in $\RR$.
Whitney's theorem establishes a convergence 
characterization for a local polynomial approximation when the degree $r-1$ of polynomials is fixed and the 
length $\delta:= b-a$ of the interval $I$ is small.
Namely, if $0 < p \le \infty$, we have for every $f \in L_p(I)$,
\begin{equation} \nonumber
C'\, \omega_r(f,\delta)_{p,I} 
\ \le \ 
E_r(f)_{p,I}  
\ \le \ 
C\, \omega_r(f,\delta)_{p,I}
\end{equation} 
with constant $C, C'$ depending only on $r$ and $p$. 
This result was first proved by Whitney 
\cite{Wh57} for $p=\infty$ and then extended by Brudny\u{\i} \cite{Br64},  to $1 \le p <\infty$ and by Storozhenko \cite{S} to $0 < p < 1$. Whitney's theorem was generalized for multivariate isotropic approximations in 
\cite{Br70}, \cite{Br70_2}, \cite{StOs78} and other. We refer the reader to \cite{HeNe07}, \cite{DU11} for surveys on univariate and multivariate Whitney's theorem and recent achievements on this topic.

The present paper is a continuation of the paper \cite{DU11}. In the latter one, Dinh D\~ung and T. Ullrich have proven a multivariate Whitney's theorem for
the local anisotropic  polynomial approximation in $L_p(Q)$ for $1 \le p \le \infty$,
where $Q$ is a $d$-parallelepiped in $\RR^d$ with sides parallel to the coordinate axes. They considered the 
error of best approximation of a function $f$ by algebraic polynomials of fixed degree at most $r_i - 1$ in variable $x_i,\ i=1,...,d$. The convergence rate of the approximation error when the size of $Q$ going to $0$ is characterized by a so-called total mixed modulus of smoothness. The method of proof in \cite{DU11} based on application of  a technique in \cite{JS76}, is not suitable to the case $0 <p<1$. In this paper, by a different method we prove this theorem for $0< p \le \infty$. 

To formulate the main result of the present paper we preliminarily introduce some necessary notations.
As usual, $\mathbb{N}$ is reserved for the natural numbers, by
$\mathbb{Z}$ we denote the set of all integers, and by $\RR$ the real numbers. 
Furthermore, $\ZZ_+$ and $\RR_+$ denote the set of non-negative integers and real numbers, respectively. 
Elements $x$ of $\RR^d$ are denoted by $x = (x_1,...,x_d)$. For a domain $D \subset {\RR}^d$, let  
$L_p(D)$, $0<p\leq \infty$, be the quasi-normed space 
of functions on $D$ with the usual $p$-th integral quasi-norm 
$\|\cdot\|_{p,D}$ to be finite if $0 < p < \infty$, whereas we use the ess sup norm if $p = \infty$. 

For $r \in {\NN}^d$, denote by $\Pp_r$  the set of algebraic polynomials of degree at most $r_i - 1$ 
at variable $x_i, \ i \in [d]$, where $[d]$ stands for the natural numbers from $1$ to $d$. 
We are interested in the $L_p$-approximation of a function
$f \in L_p(Q)$ defined on a $d$-parallelepiped 
$$
      Q := [a_1,b_1]\times ... \times [a_d,b_d]
$$
by polynomials from $\mathcal{P}_r$. 
The error of the best approximation of $f \in L_p(Q)$ by polynomials from $\Pp_r$
is measured by
\begin{equation*}  
E_r(f)_{p,Q} := \inf_{\varphi \in \Pp_r} \ \|f -\varphi\|_{p,Q}. 
\end{equation*}

For $r \in {\ZZ}_+$, $h\in \RR$, and a univariate functions $f$, the 
$r$th difference operator $\Delta_h^r(f)$ is defined by 
\begin{equation*}
\Delta_h^r(f,x) := \
\sum_{j =0}^r (-1)^{r - j} \binom{r}{j} f(x + jh)\quad,\quad \Delta_h^0(f,x) := f(x).
\end{equation*}
For $r \in {\ZZ}^d_+$, $h\in \RR^d$ and a $d$-variate function $f:\RR^d \to \RR$,  
the mixed $r$th difference operator $\Delta_h^r$ is defined by 
\begin{equation*}
\Delta_h^r := \
\prod_{i = 1}^d \Delta_{h_i}^{r_i},
\end{equation*}
where the univariate operator
$\Delta_{h_i}^{r_i}$ is applied to the univariate function $f$ by considering $f$ as a 
function of  variable $x_i$ with the other variables held fixed. 
Let
\begin{equation*}
\omega_r(f,t)_{p, Q}:= \sup_{|h_i| \le t_i, i \in [d]} 
\|\Delta_h^r(f)\|_{p,Q_{rh}}, \  t \in {\RR}^d_+,
\end{equation*} 
be the mixed $r$th modulus of smoothness of $f$,
where for $y,h \in \RR^d$, we write $yh:= (y_1h_1,...,y_dh_d)$ and
 $Q_y := \{x \in Q: x_i, x_i + y_i \in [a_i,b_i], \  i \in [d]\}$.
For $r \in {\ZZ}^d_+$ and $e \subset [d]$, 
denote by $r(e) \in {\ZZ}^d_+$ the vector with $r(e)_i = r_i, i \in e$ and 
$r(e)_i = 0, i \notin e$ \ ($r(\varnothing) = 0$).
We define the {\em total mixed modulus of smoothness of order $r$} \cite{Da96},\cite{DU11} by
\begin{equation*} 
\Omega_r(f,t)_{p,Q}
:= \ 
 \sum_{e \subset [d], e \ne \varnothing} \omega_{r(e)}(f,t)_{p,Q}, \ t \in {\RR}^d_+.
\end{equation*}
Let the size $\delta(Q)$ of $Q$ be defined as $\delta(Q):= (b_1-a_1,...,b_d-a_d) \in {\RR}^d_+$. 

The main result of the present paper is read as follows.

\begin{theorem} \label{Theorem: Whitney}
Let $0 < p \le \infty$, $r \in {\NN}^d$. Then there are constants $C, C'$ depending 
only on $r,d, p$ such that for every $f \in L_p(Q)$,
\begin{equation} \label{ineq:WhitneyIneq} 
C' \Omega_r(f,\delta)_{p,Q}
\ \le \ 
E_r(f)_{p,Q}  
\ \le \ 
C \Omega_r(f,\delta)_{p,Q}, 
\end{equation}
where $\delta = \delta(Q)$ is the size of $Q$. 
\end{theorem}

Theorem \ref{Theorem: Whitney} shows that for $0 < p \le \infty$, $r \in {\NN}^d$, the total mixed modulus of smoothness $\Omega_r(f,t)_{p,Q}$ completely characterizes the convergence rate of the best anisotropic polynomial 
$L_p$-approximation when the degree $r$ of polynomials is fixed and the size $\delta(Q)$ of $Q$ is going to $0$. It may have applications in multivariate approximations of functions with bounded mixed smoothness or differences by piecewise polynomials or splines.

Theorem \ref{Theorem: Whitney} extends  a result of the paper \cite{DU11} proving it for $1 \le p \le \infty$, and is a multivariate generalization of a result of the paper \cite{S} proving it for $d=1$ and $0<p<1$. In the latter paper, to prove her result the author used inductive arguments and Marchaud's inequality \cite{Ma27} to reduce the problem to the approximation of functions by constants.  We will develop this method to prove Theorem \ref{Theorem: Whitney}.  It turns out that in its proof we should overcome certain difficulties by employing some auxiliary results in particular, a version of  Marchaud's inequality for mixed modulus of smoothness, an upper bound for the error of the anisotropic 
approximation by Taylor polynomials for functions from Sobolev spaces of mixed smoothness, a basic relationship between the $r$th mixed difference operators $\Delta_h^r$ and the polynomials from $\Pp_r$ 
(see Lemma \ref{Lemma:a.e.polynomial}),  etc. 

It is worth to notice that there was another proof of Whitney's theorem for $d=1$ and $0<p<1$ given in \cite{SK95} based on a technique used in the original proof of Whitney \cite{Wh57}: he estimated the deviation of  the function from an interpolating polynomial with equally spaced nodes by means of finite differences. It is interesting to develop it to prove Theorem \ref{Theorem: Whitney}. However, this would go beyond the scope of the present paper.

The paper is organized as follows. 
In Section \ref{Marchaud inequality}, we prove Marchaud's inequality for mixed modulus of smoothness and other auxiliary facts. In Section \ref{Whitney's inequality} we prove Theorem \ref{Theorem: Whitney}.

\section{Marchaud's inequality and other auxiliary results} 
\label{Marchaud inequality}

The total mixed modulus of smoothness $\Omega_r(f,t)_{p,Q}$ is not suitable when we want to estimate the error of anisotropic approximations by 
by polynomials from $\mathcal{P}_r$. We therefore introduce a modification as follows.
For $r \in {\ZZ}^d_+$, $h\in \RR^d$ and a 
$d$-variate function $f:\RR^d \to \RR$, 
the mixed $p$-mean modulus of smoothness of order $r(e)$ is given by
\begin{equation*}
w_r(f,t)_{p,Q} 
\ := \ 
\left( \Big(\prod_{i=1}^d t_i^{- 1}\Big) \int_{U(t)} \int_{Q_{rh}}
|\Delta_h^r(f,x)|^p \ dx \, dh \right)^{1/p}, \ t \in \RR^d_+,
\end{equation*}
where $U(t):= \{ h \in {\RR}^d: |h_i| \le t_i, \ i \in [d]\}$, 
with the usual change of the outer mean integral to sup if $p= \infty$. This leads to 
the definition of the {\em total mixed $p$-mean modulus of smoothness of order $r$} by
\begin{equation*} 
W_r(f,t)_{p,Q}
:= \ 
 \sum_{e \subset [d], e \ne \varnothing} w_{r(e)}(f,t)_{p,Q}, \ t \in {\RR}^d_+.
\end{equation*}

We use letters $C, C', C_1, C_2,...$ to denote a positive constant independent of the parameters and/or functions 
which are relevant in the context.

\begin{lemma} \label{lemma:[W_r><Omega_r]}
Let $0 < p \le \infty$, $r \in {\NN}^d$. Then there are constants $C, C'$ depending 
only on $r,d$ such that for every $f \in L_p(Q)$,
\begin{equation*}  
C W_r(f,t)_{p,Q} 
\ \le \ 
\Omega_r(f,t)_{p,Q}  
\ \le \ 
C' W_r(f, t)_{p,Q}, \ t \in {\RR}^d_+. 
\end{equation*}
\end{lemma}

\begin{proof} It is enough to show that for  $r \in \ZZ^d_+$,
\begin{equation}  
C w_r(f,t)_{p,Q} \label{w_r><omega_r}
\ \le \ 
\omega_r(f,t)_{p,Q}  
\ \le \ 
C' w_r(f, t)_{p,Q}, \ t \in {\RR}^d_+. 
\end{equation}
The first inequality in \eqref{w_r><omega_r} follows directly from the definitions of 
$\omega_r(f,t)_{p,Q} $ and $ w_r(f,t)_{p,Q}$. For simplicity let us prove the second one for $d=2$. This inequality was proven in \cite{PP87} for the univariate case $(d=1)$. Therefore, we have for $|h_i| \le t_i$ and for almost all $x_{i'} \in \II_{i'}$,
\begin{equation} \nonumber 
\| \Delta_{h_i}^{r_i}(f)\|_{p,x_i}^p  
\ \le \ C_1\,
t_i^{-1} \int_{U(t_i)} \int_{\II_i(r_ih_i)}
|\Delta_{h_i}^{r_i}(f,x)|^p \ d{x_i} \, d{h_i}, \ \ i=1,2,
\end{equation}
where $i'=2$ if $i=1$, and $i'=1$  if $i=2$, and the quasi-norm
$\| \Delta_{h_i}^{r_i}(f)\|_{p,x_i}$ is applied to the  function $f$ by considering $f$ as a univariate
function in  variable $x_i$ with the other variable held fixed. Hence, by using of the identity 
$\| \Delta_h^r(f)\|_p = \|\|\Delta_{h_2}^{r_2} (\Delta_{h_1}^{r_1}(f))\|_{p,x_1}\|_{p,x_2}$
and Fubini's theorem we prove the second inequality in \eqref{w_r><omega_r}.
\end{proof}

\begin{lemma} \label{Lemma[IneqW_r(f,t)]}
If $\Bbb I=\bigcup_{j=1}^n\Bbb I_j$, where $\Bbb I_j$ are cubes with disjoint interiors, $j=1,...,n$, then there is a constant $C$ depending only on $r,d,p$ such that
\begin{equation}\nonumber
\sum_{j=1}^n W_r(f,t)_{p,\Bbb I_j}^p \le CW_r(f,t)_{p,\Bbb I}^p.
\end{equation}
\end{lemma}

\begin{proof}
Indeed, we have
\[ 
\sum_{j=1}^n W_r(f,t)_{p,\Bbb I_j}^p \le C \sum_{j=1}^n \sum_{e \subset [d], \, e  \ne \varnothing}
w_{r(e)}(f,t)_{p,\Bbb I_j}^p=C\sum_{e \subset [d], \, e  \ne \varnothing} \ \sum_{j=1}^n w_{r(e)}(f,t)_{p,\Bbb I_j}^p,
\]
and
\begin{align*} 
\sum_{j=1}^n w_{r(e)}(f,t)_{p,\Bbb I_j}^p&= \prod_{i=1}^d  t_i^{-1}\int_{U(t)}\ \sum_{j=1}^n\int_{(\Bbb I_j)_{r(e)h}}
|\Delta^{r(e)}_h(f,x)|^pdx\,dh\\[2ex]
& \le \ 
\prod_{i=1}^d  t_i^{-1}\int_{U(t)}\int_{\Bbb I_{r(e)h}}|\Delta^{r(e)}_h(f,x)|^pdx\,dh\\[2ex]
& \le \ 
w_{r(e)}(f,t)_{p,\Bbb I}^p.
\end{align*}
Hence,
 \[
\sum_{j=1}^n W_r(f,t)_{p,\Bbb I_j}^p \le C\sum_{e \subset [d], \, e \ne \varnothing}w_{r(e)}(f,t)_{p,\Bbb I}^p \le 
C' W_r(f,t)_{p,\Bbb I}^p.
\]
\end{proof}

The following Marchaud's inequality for mixed modulus of smoothness gives upper bounds of a mixed modulus of smoothness of a function on $d$-parallelepiped $Q$ by its higher order's mixed modulus of smoothness and $L_p$-quasi-norm (for $d=1$ see \cite{Ma27} and also \cite[Theorems II.8.1 \& II.8.2]{DL}). It allows us to reduce the general case of Theorem \ref{Theorem: Whitney} to the simplest case where $r=(1,1,...,1)$ and therefore we deal with the 
$L_p$-approximation by constant functions.

\begin{lemma} \label{Lemma[MarchaudIneq]}
Let $0<p \le \infty$, $i \in [d]$, $\delta(Q):=(\delta_1,...,\delta_d)$ be the size of $Q$ and $k,r \in \Bbb Z_+^d$ such that $1 \le k_i<r_i$ and  $k_j=r_j$ $j\neq i$. Then there is a constant $C$ depending only on $r,d$ and $p$ such that for each $f \in L_p(Q),\ t=(t_1,..,t_d)>0$,
\begin{equation} \label{MarchaudIneq(p>1)}
\omega_k(f,t)_{p,Q} \le C\,t_i^{k_i} 
\left[ \int_{t_i}^{\delta_i}\dfrac{\omega_r(f,(t_1,...,t_{i-1},u,t_{i+1},...,t_d))_{p,Q}\,du}{u^{k_i+1}}
\ + \ \dfrac{\|f\|_{p,Q}}{\delta_i^{k_i}}\right], \ 1 \le p \le \infty, 
\end{equation}
and 
\begin{equation} \label{MarchaudIneq}
\omega_k(f,t)_{p,Q}^p \le C^p\,t_i^{pk_i} 
\left[ \int_{t_i}^{\delta_i}\dfrac{\omega_r(f,(t_1,...,t_{i-1},u,t_{i+1},...,t_d))_{p,Q}^p\,du}{u^{k_ip+1}}
\ + \ \dfrac{\|f\|_{p,Q}^p}{\delta_i^{pk_i}}\right], \ 0 < p < 1. 
\end{equation}
\end{lemma}

\begin{proof}
We will prove the inequality \eqref{MarchaudIneq}. The inequality \eqref{MarchaudIneq(p>1)} can be proven in a similar way with a slight modification. For simplicity let us prove for $d=2$ when $Q=[a_1,b_1] \times [a_2,b_2]$. The case $d>2$ can be proven analogously by induction on $d$. Set $Q_i:=[a_i,b_i]$, $i=1,2$.  

We first prove  \eqref{MarchaudIneq} for the special case $r=(k_1+1,k_2)$. Let the shift operator $T_h$ be defined by  $T_h(f,x):=f(x+h)$. From the identity 
\[
(x-1)^{k_1}=2^{-k_1}(x^2-1)^{k_1}+P(x)(x-1)^{k_1+1}
\]
 for the polynomial $P(x):=[1-2^{-k_1}(x+1)^{k_1}]/(x-1)$ of degree $k_1-1$, we have
\begin{align} \label{83}
(T_{h_1}-I)^{k_1}=2^{-k_1}(T_{2h_1}-I)^{k_1}+P(T_{h_1})(T_{h_1}-I)^{k_1+1}. 
\end{align}
Let $Q_1':=[a_1,c_1]$, $c_1:=(a_1+b_1)/2$. 
Notice that $\|P(T_{h_1})(g)\|_p \le M\|g\|_p$ for $g \in L_p(Q_1')$. Hence, by\eqref{83} we obtain
\begin{align}
 \|\Delta^{k_1}_{h_1} \Delta^{k_2}_{h_2}(f)\|_{p,Q_1'}^p 
& \nonumber \le 2^{-pk_1}\|\Delta^{k_1}_{2h_1} \Delta^{k_2}_{h_2}(f)\|_{p,Q_1'}^p+
M^p\|\Delta^{k_1+1}_{h_1}\Delta^{k_2}_{h_2}(f)\|_{p,Q_1'}^p\\[2ex]
& \le \ 
M^p\sum_{j=0}^{m}2^{-k_1jp}\|\Delta^{k_1+1}_{2^jh_1}\Delta^{k_2}_{h_2}(f)\|_{p,Q_1'}^p+
2^{-k_1(m+1)p}\|\Delta^{k_1}_{2^{m+1}h_1}\Delta^{k_2}_{h_2}(f)\|_{p,Q_1'}^p \\[2ex]
&\label{85} 
\le \ 
M^p\sum_{j=0}^{m}2^{-k_1jp}\|\Delta^{k_1+1}_{2^jh_1}\Delta^{k_2}_{h_2}(f)\|_{p,Q_1}^p+
2^{-k_1(m+1)p}\|\Delta^{k_1}_{2^{m+1}h_1}\Delta^{k_2}_{h_2}(f)\|_{p,Q_1}^p 
\end{align}
provided that $2^{m+2}k_1h_1 \le \delta_1$. Here we consider $f$ as a function of variable $x_1$ with $x_2$ held fixed.
We see that \eqref{85} holds also if $Q_1'$ is replaced by $Q_1'':=[c_1,b_1]$. This can be obtained by applying \eqref{85} to the function
$g(x):=f(b_1-x)$ which has the same moduli of smoothness as $f$. Therefore, it follows that \eqref{85} also holds with $Q_1$ in place of $Q_1'$ and with additional multiplier 2 in the right-hand side:
$$
\|\Delta^{k_1}_{h_1} \Delta^{k_2}_{h_2}(f)\|_{p,Q_1}^p \le 
2M^p\sum_{j=0}^{m}2^{-k_1jp}\|\Delta^{k_1+1}_{2^jh_1}\Delta^{k_2}_{h_2}(f)\|_{p,Q_1}^p
\ + \ 
2.2^{-k_1(m+1)p}\|\Delta^{k_1}_{2^{m+1}h_1}\Delta^{k_2}_{h_2}(f)\|_{p,Q_1}^p.
$$
For $h=(h_1,h_2) \in \RR^2, \ t=(t_1,t_2) \in \RR^2_+,$ with $|h_1| \le t_1,\ |h_2| \le t_2$, from the last inequality we derive
\begin{align*}
\int_{Q_2}\int_{Q_1}|\Delta^{k_1}_{h_1}\Delta^{k_2}_{h_2}(f)|^p \, dx_1\, dx_2 
& \le 
2M^p \sum_{j=0}^m2^{-k_1jp}\int_{Q_1}\int_{Q_2}
|\Delta^{k_2}_{h_2}\Delta^{k_1+1}_{2^jh_1}(f)|^p \, dx_2\, dx_1\\[2ex]
& + 2.2^{-k_1(m+1)p} \int_{Q_1}\int_{Q_2}
|\Delta^{k_2}_{h_2}\Delta^{k_1}_{2^{m+1}h_1}(f)|^p \, dx_2\, dx_1\\[2ex]
& \le C' \bigg[t_1^{k_1p} \sum_{j=0}^m (2^{pj}t_1^p)^{-k_1}\int_{Q_1}\int_{Q_2}
|\Delta^{k_2}_{h_2}\Delta^{k_1+1}_{2^jh_1}(f)|^p \, dx_2\, dx_1\\[2ex]
&+2^{-k_1mp}2^{k_2p}\|f\|_{p,Q}^p \bigg].
\end{align*}
Since
$$
(2^{pj}t_1^p)^{-k_1} \int_{Q_1}\int_{Q_2} |\Delta^{k_2}_{h_2}\Delta^{k_1+1}_{2^jh_1}(f)|^p \, dx_2dx_1
\ \le \ 
2^{k_1p + 1}\int_{2^jt_1}^{2^{j+1}t_1}
\omega_{k_1+1,k_2}(f,(u,t_2))_{p,Q}^pu^{-k_1p-1}du, 
$$
if we take $m$ to be the last integer for which $2^{m+2}k_1t_1 \le \delta_1$, then exists a constant $C$ such that
$$
\omega_k(f,t)_{p,Q}^p \le C^p t_1^{k_1p}\left[\int_{t_1}^{\delta_1}
\dfrac{\omega_{k_1+1,k_2}(f,(u,t_2))_{p,Q}^p}{u^{k_1p+1}}du +
\dfrac{\|f\|_{p,Q}^p}{\delta_1^{k_1p}}\right].
$$
Thus,  the inequality \eqref{MarchaudIneq} has been proven for the case $r=(k_1+1,k_2)$. 

We now prove it for arbitrary $k,r$ by induction on $r_1$. We assume that \eqref{MarchaudIneq} holds true for  $r=(r_1,k_2)$ and prove it for $r=(r_1+1,k_2)$. Hence, by the inequality \eqref{MarchaudIneq} for $r=(k_1+1,k_2)$
we get
\begin{equation} \nonumber
\begin{split}
\omega_k(f,t)_{p,Q}^p & \le  C_1t_1^{k_1p}\left[\int_{t_1}^{\delta_1}\dfrac{\omega_{r_1,k_2}(f,(u,t_2))_{p,Q}^p}{u^{k_1p+1}}du 
\ + \ 
\dfrac{\|f\|_{p,Q}^p}{\delta_1^{k_1p}} \right]\\[2ex]
& \le C_1t_1^{k_1p}\left[\int_{t_1}^{\delta_1}u^{-k_1p-1}C_2u^{r_1p}\left(\int_u^{\delta_1}
\dfrac{\omega_{r_1+1,k_2}(f,(h,t_2))_{p,Q}^p}{h^{r_1p+1}}\, dh 
\ + \ 
\dfrac{\|f\|_{p,Q}^p}{\delta_1^{r_1p}} \right)du 
\ + \ 
\dfrac{\|f\|_{p,Q}^p}{\delta_1^{k_1p}}\right]\\[2ex]
& \le \
C_1t_1^{k_1p}\dfrac{\|f\|_{p,Q}^p}{\delta_1^{k_1p}}+C_3 t_1^{k_1p}\dfrac{\|f\|_{p,Q}^p}{\delta_1^{r_1p}}\int_{t_1}^{\delta_1} u^{r_1p-k_1p-1}du\\[2ex]
& +C_3t_1^{k_1p}\int_{t_1}^{\delta_1}\dfrac{\omega_{r_1+1,k_2}(f,(h,t_2))_{p,Q}^pdh}{h^{r_1p+1}}dh \int_{t_1}^h u^{r_1p-k_1p-1}\, du\\[2ex]
& \le \ 
C_4t_1^{k_1p} \left[\int_{t_1}^{\delta_1}\dfrac{\omega_{r_1+1,k_2}(f,(h,t_2))_{p,Q}^pdh}{h^{k_1p+1}}dh +\dfrac{\|f\|_{p,Q}^p}{\delta_1^{k_1p}} \right].
\end{split}
\end{equation}
\end{proof}

By $f^{(k)}, \ k \in {\ZZ}^d_+$, we denote the $k$-th order 
generalized mixed  derivative of
a locally integrable function $f$ ,
i.e., 
$$
        \int_Q f^{(k)}(x)\varphi(x)\,dx = (-1)^{k_1+...+k_d} \int_Q f(x)
\frac{\partial^{k_1+...+k_d}\varphi}{\partial x_1^{k_1}\cdots \partial
x_d^{k_d}}(x)\,dx
$$
for all test functions $\varphi \in C_0^{\infty}(Q)$, where $C_0^{\infty}(Q)$
is the space of infinitely differentiable functions on $Q$ with compact support which is
interior to $Q$.
For $r \in {\ZZ}^d_+$ and $0 < p\leq \infty$, the  Sobolev 
space $W^r_p(Q)$ of mixed smoothness $r$ is defined as the set of locally integrable
functions $f \in L_p(Q)$, for which the generalized derivative $f^{(r(e))}$ exists
as a locally integrable function and the following quasi-norm is finite
\begin{equation*}
\|f\|_{W^r_p(Q)}
:= \ 
 \sum_{e \subset [d]} \|f^{(r(e))}\|_{p,Q}.
\end{equation*}

For $x,y \in {\RR}^d$, the inequality $x \le y \ (x < y)$ means that $x_j \le y_j \ (x_j < y_j), \ j \in [d]$. Let $f\in W^r_p(Q)$ and $x^0 \in Q$. Then $f$ has continuous derivatives of order $k$ for 
each $k < r$, therefore,  we can define the Taylor polynomial $P_k(f)$ of order $k$ by
 \begin{equation*}  
P_k(f,x)  
\ := \ 
P_k(f,x^0,x)  
\ = \ 
\sum_{0 \le s < k} f^{(s)}(x^0)e_s(x-x^0), 
\end{equation*}
where $e_s(x) := \prod_{i=1}^d e_{s_i}(x_i)$ and $e_m(t) := t^m/m!$. 

We will need an estimate of the error of the approximation of a function $f\in W^r_p(Q)$  by Taylor polynomials via the size of $Q$ and the $L_p$-quasi-norm of its derivatives. The following lemma on multivariate anisotropic Taylor polynomial approximation is a nontrivial generalization of the well-known univariate result. For a proof of this lemma we refer the reader to \cite{DU11}.

\begin{lemma} \label{Lemma:[E_{r-1}(f)<]}
Let $1 \le p \le \infty$, $r \in {\NN}^d$. Then there is a constant $C$ depending 
only on $r,d$ such that for every $f \in W^r_p(Q)$,
\begin{equation*}  
E_r(f)_{p,Q}  
\ \le \
\|f - P_r(f)\|_{p,Q}  
\ \le \ 
C \sum_{e \subset [d], \, e \ne \varnothing} \ \prod_{i \in e} \delta_i^{r_i} \|f^{(r(e))}\|_{p,Q}, 
\end{equation*}
where $\delta = \delta(Q)$ is the size of $Q$.
 \end{lemma}
  
We denote by $L_p^{\operatorname{loc}}(\RR^d)$ the set of all functions $f$ on $\RR^d$ such that for every 
$x \in \RR^d$, there exists a neighborhood $V_x$ of $x$ such that $f \in L_p(V_x)$.  A basic property of the univariate $r$th difference operator $\Delta_h^r$ is that it turns a function $f \in L_p^{\operatorname{loc}}(\RR)$ to an almost everywhere zero function, i.e., $\Delta_h^r(f, x)= 0$  for almost every $x$, if and only if $f$ is almost everywhere equal to a polynomial from $\Pp_r$ (see, e.g., \cite[Proposition II.7.1]{DL}). The following lemma generalizes this property to the multivariate mixed difference operator  $\Delta_h^r$.     

\begin{lemma} \label{Lemma:a.e.polynomial}
Let $0 < p \le \infty$, $r \in {\NN}^d$ and $f \in L_p^{\operatorname{loc}}(\RR^d)$. 
Then $f$ is almost everywhere equal to a polynomial in $\Pp_r$ if and only if 
\begin{equation} \label{eq:[Delta_h^{r(e)}(f,x)]}
\Delta_h^{r(e)}(f,x) = 0 
\end{equation}
for all non-empty $e \subset [d]$ and for 
almost all $x,h \in \RR^d$. 
 \end{lemma}
 
\begin{proof}
The statement "only if" can be verified directly. Let us prove the other one. 
We need the following auxiliary fact. There are coefficients 
$a_k, \ 0 < k \le r$ and $b_e, \ e \subset [d], e \ne \varnothing$, such that
\begin{equation} \label{eq:[1]}
1 \ = \
\sum_{0 < k \le r} a_k x^k + \sum_{e \subset [d], \, e \ne \varnothing} b_e P_e(x)
\end{equation}
where $x^k:= \prod_{j=1}^d x_j^{k_j}$ and
\[P_e(x):= \ \prod_{i \in e} (x_i - 1)^{r_i}.
\]
Indeed, putting 
\[
A_e(x):= \ \prod_{i \in e} [(x_i - 1)^{r_i} - (-1)^{r_i}],
\]
we have
\begin{equation*} 
\begin{aligned}
P_{[d]}(x) 
\ = \
\prod_{i=1}^d (x_i - 1)^{r_i}
\ & = \
\prod_{i=1}^d [(x_i - 1)^{r_i} - (-1)^{r_i} + (-1)^{r_i}] \\
\ & = \
\sum_{e \subset [d]} \ \prod_{i \in [d] \setminus e}(-1)^{r_i} \ \prod_{i \in e}[(x_i - 1)^{r_i} - (-1)^{r_i}] \\
\ & = \
A_{[d]}(x) + \prod_{i=1}^d (-1)^{r_i} +
\sum_{e \in [d], \, e \ne [d], \varnothing} \ \prod_{i \in [d] \setminus e}(-1)^{r_i} \ A_e(x), \\
\end{aligned}
\end{equation*}
and
\begin{equation*} 
A_e(x) 
\  = \
\sum_{u \subset e, \, u \ne \varnothing} \ \prod_{i \in e \setminus u}(-1)^{r_i} \ P_u(x) 
+ (-1)^{|e|} \prod_{i \in e} (-1)^{r_i}.
\end{equation*}
Hence, it is easy to verify that
\begin{equation*} 
P_{[d]}(x) 
\  = \
A_{[d]}(x)
 + \sum_{e \subset [d], \, e \ne [d], \varnothing} \ \prod_{i \in [d] \setminus e}(-1)^{r_i}  
\sum_{u \subset e, \, u \ne \varnothing} \ \prod_{i \in e \setminus u}(-1)^{r_i} \ P_u(x) 
- (-1)^d \prod_{i=1}^d (-1)^{r_i}, 
\end{equation*}
or equivalently,
\begin{equation*} 
1
\  = \
(-1)^d \prod_{i=1}^d (-1)^{r_i}\left\{
A_{[d]}(x)
+ \sum_{e \subset [d], \, e \ne [d], \varnothing} \   
\sum_{u \subset e, \, u \ne \varnothing} \ 
\prod_{i \in [d] \setminus u}(-1)^{r_i} \ P_u(x) - P_{[d]}(x)\right\}. 
\end{equation*}
Notice that the polynomial in the right-hand side of the last equality is 
of the form \eqref{eq:[1]} what is desired. 

For a non-negative integer $s$ and a function $f$ on $\RR$, 
let the operator $T_h^s, \ h \in \RR$, be defined by $T_h^s(f,x):= f(x + sh)$. 
For a $k \in \ZZ^d_+$ and a function $f$ defined on $\RR^d$, 
let the mixed operator $T_h^k, \ h \in \RR^d$, be defined by 
\begin{equation*}
T_h^k(f):= \prod_{i=1}^d T_{h_i}^{k_i}(f), 
\end{equation*}
where the univariate operator $T_{h_i}^{k_i}$ is applied to $f$ as a univariate function in variable $x_i$ with the other variables held fixed. 
By using the correspondence  between the operator $T_h^k$ and the monomial $x^k$, and
between the operator $\Delta_h^{r(e)}$ and the polynomial $P_e(x)$, from \eqref{eq:[1]} we 
get the following equality for a function $f$ defined on $\RR^d$,
\begin{equation*} 
f(x) \ = \
\sum_{0 < k \le r} a_k f(x + kh) + \sum_{e \in [d], \, e \ne \varnothing} b_e \Delta_h^{r(e)}(f,x).
\end{equation*}
Let $f \in L_p^{\operatorname{loc}}(\RR^d)$ satisfying 
the condition \eqref{eq:[Delta_h^{r(e)}(f,x)]}. Then we have for almost all $x,h \in \RR^d$,
\begin{equation} \label{eq:[f]}
f(x) \ = \
\sum_{0 < k \le r} a_k f(x + kh).
\end{equation}
We first let $1 \le p \le \infty$. 
Take a function $g \in C^\infty_0(\RR^d)$ with $\int_{\RR^d}g(h) \, dh = 1$.
Then by \eqref{eq:[f]} we have for almost all $x$,
\begin{equation*} 
\begin{aligned}
f(x) \ = \
f(x)\int_{\RR^d}g(h) \, dh
\ & = \
\sum_{0 < k \le r} a_k \int_{\RR^d}g(h)f(x + kh) \, dh \\
\ & = \
\sum_{0 < k \le r} a_k \prod_{i=1}^d k_i^{-1}\int_{\RR^d}g((y-x)/k)f(y) \, dy,
\end{aligned}
\end{equation*}
where $(y-x)/k = ((y_1 - x_1)/k_1,...,(y_d - x_d)/k_d)$. Notice that each term in the right-hand side 
is a function in  $C^\infty(\RR^d)$. Hence, after redefinition on a set of measure zero $f \in C^\infty(\RR^d)$.
From \eqref{eq:[f]} and the equality 
\begin{equation*}
f^{(k)}(x)
\ = \
\lim_{h \to 0} \ \prod_{i=1}^d h_i^{-k_i} \Delta_h^k(f,x)
\end{equation*}
for any $k \in \ZZ^d_+$ we conclude that $f^{(r(e))}(x)= 0$ for all non-empty $e \subset [d]$. 
Applying Lemma \ref{Lemma:[E_{r-1}(f)<]} gives for all $N > 0$,
\begin{equation*}  
\|f - P_r(f)\|_{p,Q_N}  
\ \le \ 
C \sum_{e \subset [d], e \ne \varnothing} \ \prod_{i \in e} N^{r_i} \|f^{(r(e))}\|_{p,Q_N} 
\ = \ 0, 
\end{equation*}
where $Q_N$ the $d$-cube of the size $N$. This implies that 
$f(x) = P_r(f,x) \in\Pp_r$ almost everywhere on $\RR^d$. 
We now consider the case $0 < p < 1$. By \eqref{eq:[f]} and Jensen's inequality 
we have for almost all $x,h \in \RR^d$,
\begin{equation*} 
|f(x)|^p \ \le \
\sum_{0 < k \le r} |a_k|^p |f(x + kh)|^p.
\end{equation*}
Hence,
\begin{equation*} 
|f(x)|^p \ \le \
\sum_{0 < k \le r} |a_k|^p \int_{|h| \le 1}|f(x + kh)|^p \, dh.
\end{equation*}
Since the right-hand side is a locally bounded function on $\RR^d$, so is $f$.
From the proven case $p = \infty$ it follows that 
$f(x) = P_r(f,x) \in\Pp_r$ almost everywhere on $\RR^d$.
\end{proof}

\section{Proof of Theorem  \ref{Theorem: Whitney}}
\label{Whitney's inequality}

In this section we prove Theorem \ref{Theorem: Whitney}. To the end we need two auxiliary lemmas more. The first lemma gives an upper bound of the error of the approximation of a function $f \in L_p(Q)$ by constants functions. The second one establishes sufficient conditions of pre-compactness of a subset in $L_p(Q)$ via the total mixed modulus of smoothness of $\Omega_r(f,t)_{p,Q}$.

\begin{lemma} \label{ApproxbyConst}
Let $0<p \le \infty$. Then there is a constants $C$ depending only on $r,d,p$ such that for every $f \in L_p(Q)$, exists a constant $\beta$ such that
$$
 \|f-\beta\|_{p,Q} \le C \Omega_{\bf 1}(f,\delta)_{p,Q}, 
$$
where $\delta = \delta(Q)$ is the size of $Q$ and $\bf 1$ $=(1,..,1) \in \Bbb N^d.$
\end{lemma}

\begin{proof}
Let us prove for the case $0<p \le 1$ and $d=2$. The general case can be proven in a similar way with a slight modification. 
Recall that  $Q=[a_1,b_1] \times [a_2,b_2]$ and 
$\delta(Q)= (\delta_1, \delta_2)= (b_1 - a_1, b_2 - a_2)$ for $d=2$. We have
\begin{align*} 
&\int_{Q}\int_{Q}|f(x)-f(y)|^p\,dx\,dy \\[2ex]
&\le \int_{Q}\int_{Q}|f(x_1,x_2)-f(y_1,x_2)|^p\,dx\,dy
\ + \ \int_{Q}\int_{Q}|f(y_1,x_2)-f(y_1,y_2)|^p\,dx\,dy \\[2ex]
&\le \delta_2\int_{a_2}^{b_2}\int_{a_1}^{b_1}\int_{a_1}^{b_1}|f(x_1,x_2)-f(y_1,x_2)|^p\,dx_1\,dy_1\,dx_2\\[2ex]
&+\delta_1 \int_{a_1}^{b_1}\int_{a_2}^{b_2}\int_{a_2}^{b_2}|f(y_1,x_2)-f(y_1,y_2)|^p\,dx_2\,dy_2\,dy_1
=: I_1 \ + \ I_2.
\end{align*}
On the other hand, 
\begin{align*} 
&\int_{a_1}^{b_1}\int_{a_1}^{b_1}|f(x_1,x_2)-f(y_1,x_2)|^p\,dx_1\,dy_1 \\[2ex]
&= \int_{a_1}^{b_1}\int_{a_1}^{x_1}|f(x_1,x_2)-f(y_1,x_2)|^p\,dy_1\,dx_1
\ +\ \int_{a_1}^{b_1}\int_{x_1}^{b_1}|f(x_1,x_2)-f(y_1,x_2)|^p\,dy_1\,dx_1\\[2ex]
&= \int_0^{\delta_1}\int_{a_1+u}^{b_1}|f(x_1,x_2)-f(x_1-u,x_2)|^p\,dx_1\, du
\ +\ \int_0^{\delta_1}\int_{a_1}^{b_1-u}|f(x_1,x_2)-f(x_1+u,x_2)|^p\,dx_1\, du \\[2ex]
&= 2 \int_0^{\delta_1}\int_{a_1}^{b_1-u}|f(x_1+u,x_2)-f(x_1,x_2)|^p\,dx_1\, du.
\end{align*}
Hence,
\begin{equation} \nonumber
I_1 
\ \le \ 
2 \delta_2\int_0^{\delta_1}\int_{a_2}^{b_2}\int_{a_1}^{b_1-u}|f(x_1+u,x_2)-f(x_1,x_2)|^p\,dx_1\, dx_2\, du
\ \le \ 
2 \delta_1\delta_2\omega_{(1,0)}(f,\delta(Q))_{p,Q}^p.
\end{equation}
In a similar way we prove that
\begin{equation} \nonumber
I_2 
\ \le \ 
2 \delta_1\delta_2\omega_{(0,1)}(f,\delta(Q))_{p,Q}^p.
\end{equation}
Thus, we have proven the following inequality
\begin{equation} \label{ineq[int_{Q}int_{Q}|f(x)-f(y)|]}
\int_{Q}\int_{Q}|f(x)-f(y)|^p\,dx\, dy  
\ \le \ 
2 \delta_1\delta_2
\left[ \omega_{(1,0)}(f,\delta(Q))_{p,Q}^p \ + \ \omega_{(0,1)}(f,\delta(Q))_{p,Q}^p \right].
\end{equation}
The function
\begin{equation} \nonumber
g(y)
:= \ \int_{Q}|f(x)-f(y)|^p\,dx
\end{equation} 
is continuous on $Q$. Consequently, by \eqref{ineq[int_{Q}int_{Q}|f(x)-f(y)|]} there is a $y^* \in Q$ such that for 
$\beta = f(y^*)$, 
\begin{equation} \nonumber
\int_{Q}|f(x)- \beta|^p\,dx  
\ \le \ 
2 \left[ \omega_{(1,0)}(f,\delta(Q))_{p,Q}^p \ + \ \omega_{(0,1)}(f,\delta(Q))_{p,Q}^p \right]
\ \le \ 
C \Omega_{\bf 1}(f,\delta(Q))_{p,Q}^p.
\end{equation}
\end{proof}

\begin{lemma} \label{Lemma[Precompact]}
Let $0 < p \le \infty$, $r \in {\NN}^d$ and $F$ is a set of functions in 
$L_p(Q)$. Then $F$ is pre-compact in $L_p(Q)$ if $F$ is bounded, i.e., 
$\|f\|_{p,Q} \le M$ for a constant $M$, and   
\begin{equation}  \label{Condition[Precompact]}
\lim_{t \to 0} \Omega_r(f,t)_{p,Q}
\ = \ 0, \ \text{for} \ t >0, \ t \in \RR^d_+, \ \text{uniformly for} \ f \in F.
\end{equation}
 \end{lemma}
 
 \begin{proof}
For simplicity we prove the lemma for the case $0 < p \le 1$ and $Q=[0,1]^2$. The general case can be proven in a similar way. Put $r^1=(r_1,0)$, $r^2=(0,r_2)$. By Lemma \ref{Lemma[MarchaudIneq]}
\begin{equation} \nonumber
\omega_{(1,0)}(f,t)_{p,Q}^p \le C_1 t_1^p 
\left[ \int_{t_1}^{\delta_1}\dfrac{\omega_{r^1}(f,(u,t_2))_{p,Q}^p\,du}{u^{p+1}}
\ + \ M^p\right], 
\end{equation}
\begin{equation} \nonumber
\omega_{(0,1)}(f,t)_{p,Q}^p \le C_1 t_2^p 
\left[ \int_{t_2}^{\delta_2}\dfrac{\omega_{r^2}(f,(t_1,u))_{p,Q}^p\,du}{u^{p+1}}
\ + \ M^p\right]. 
\end{equation}
Hence, by \eqref{Condition[Precompact]} and the inequality 
$\omega_{(1,1)}(f,t)_{p,Q}^p \le 2\,\omega_{(1,0)}(f,t)_{p,Q}^p$ we have that 
\begin{equation}  \label{[UniformConvergence]}
\lim_{t \to 0} \Omega_{{\bf 1}}(f,t)_{p,Q}
\ = \ 0, \ \text{for} \ t >0, \ t \in \RR^d_+, \ \text{uniformly for} \ f \in F.
\end{equation}

For $n \in \NN$, we define the space $S_n$ of all piecewise constant functions $f$ on $Q$ such that $f(x) = c_{i,j}$ for $x \in I_{i,j}:=[i/n,(i+1)/n]\times[j/n,(j+1)/n]$, where $c_{i,j}$ are a constant. 
Due to to Lemmas \ref{ApproxbyConst} and \ref{lemma:[W_r><Omega_r]} for each cube $I_{i,j}$  there is a constant $\beta_{ij}$ satisfying
\begin{align} \label{519}
 \|f-\beta_{i,j}\|_{p,I_{i,j}} \le C_2 \Omega_{{\bf 1}}(f,(1/n,1/n))_{p,I_{i,j}}  
\le C_3 W_{{\bf 1}}(f,(1/n,1/n))_{p,I_{i,j}}. 
\end{align}
\indent
Let $P_n(f,x):=\beta_{i,j}$, for $x \in I_{i,j}$. From \eqref{519} and Lemmas \ref{Lemma[IneqW_r(f,t)]} and 
\ref{lemma:[W_r><Omega_r]} we obtain  
 \begin{equation*} 
\begin{split}
 \|f-P_n(f)\|_{p,Q}^p
\ &= \  
\sum_{i,j} \|f-\beta_{i,j}\|_{p,I_{i,j}}^p 
\ \le \ 
C_3 \sum_{i,j} W_{{\bf 1}}(f,(1/n,1/n))_{p,I_{i,j}}^p \\
\ &\le \ 
C_4  W_{{\bf 1}}(f,(1/n,1/n))_{p,Q}^p 
\ \le \ 
C_5 \Omega_{{\bf 1}}(f,(1/n,1/n))_{p,Q}^p. 
\end{split} 
\end{equation*}

Hence, by \eqref{[UniformConvergence]} for arbitrary $\varepsilon>0,$  we can choose $n$ so that 
\begin{equation} \label{[|f-P_n(f)|_{p,Q}<]}
 \|f-P_n(f)\|_{p,Q} \ \le \ \varepsilon,\ f \in F. 
\end{equation} 
Moreover, since $F$ is bounded, so is the set $E:=\{P_n(f): f \in F\}$. Consequently, $E$ is pre-compact as a bounded subset in a finite dimensional subspace. Hence, there is a finite $\varepsilon$-net for $E$, which by 
\eqref{[|f-P_n(f)|_{p,Q}<]} is a finite $2^{1/p}\varepsilon$-net for $F$. This means that  $F$ is pre-compact in $L_p(Q)$.
\end{proof}

We are now able to prove
Theorem \ref{Theorem: Whitney}.

{\it Proof of Theorem  \ref{Theorem: Whitney}} \quad The first inequality in \eqref{ineq:WhitneyIneq} is trivial.
Indeed, if $f \in L_p(Q)$ then for every non-empty $e \subset [d]$ and every $\varphi \in \Pp_r$ we have
\begin{equation*} 
\omega_{r(e)}(f,\delta)_{p,Q}
\  = \ 
\omega_{r(e)}(f - \varphi, \delta)_{p,Q} 
\  \ll \ 
\|f - \varphi\|_{p,Q}.
\end{equation*}
Hence, we obtain the first inequality in \eqref{ineq:WhitneyIneq}.

Let us prove the second inequality. It is sufficient to prove it for $Q = \II^d:= [0,1]^d$. 
Suppose that it is not true. Then for each $n \in \NN$, there would exist a function $f_n \in L_p(\II^d)$ such that
\begin{equation*}
\quad E_r(f_n)_{p,\II^d} \ = \ \|f_n\|_p \ = \ 1, \quad \Omega_r(f_n, {\bf 1})_{p,\II^d} \ \le \ 1/n. 
\end{equation*}
From the convergence $\Omega_r(f_n,t)_{p,\II^d} \to 0, \ t \to 0$, for each $n$,
and the inequality $\Omega_r(f_n, t)_{p,\II^d} \le 1/n$ for all $t \in \RR^d_+, \ n \in \NN$, we can see that
this convergence is uniform in $n$. By Lemma \ref{Lemma[Precompact]} the set 
$F = \{f_n\}_{n=1}^\infty$ is precompact. Therefore, there is a subsequence $\{f_{n_k}\}_{k=1}^\infty$
such that 
\begin{equation*}
f_{n_k} \to f \in L_p(\II^d), \ k \to \infty. 
\end{equation*}
We have $\Omega_r(f, t)_{p,\II^d} = 0, \ t \in \RR^d_+$. This implies that
$ \Delta_h^{r(e)}(f,x) = 0$ 
for all non-empty $e \subset [d]$ and for 
almost all $x,h \in \RR^d$. Then by Lemma \ref{Lemma:a.e.polynomial} 
$f$ is almost everywhere equal to a polynomial in $\Pp_r$.
But this is a contradiction because 
\begin{equation*}
E_r(f)_{p,\II^d}  \ = \ \lim_{n \to \infty} E_r(f_n)_{p,\II^d} \ = \ 1. 
\end{equation*}
$\square$

{\bf Acknowledgements}
\ This research work is funded by Vietnam National Foundation for Science and Technology Development (NAFOSTED) under Grant No. 102.01-2012.15.

\end{document}